\documentclass[10pt,a4paper]{amsart}
 \usepackage{amsmath,amssymb,amsthm,latexsym,capt-of,graphicx,array}
 \input xy
\xyoption{all}
 \newtheorem{theorem}{Theorem}[section]
\newtheorem{lemma}[theorem]{Lemma}
\newtheorem{proposition}[theorem]{Proposition}
\newtheorem{corollary}[theorem]{Corollary}

\newcommand{\fp}{$\{\alpha_n\}$ }
\newcommand{\fpti}{$\{\tilde{\alpha}_i\}$ }
\newcommand{\bwip}{$\pi_{0}T(e)\text{ } $}
\newcommand{\inversedendrite}{\hat{D}_{\tau}}
\newcommand{\ray}{\Phi}

\newenvironment{example}[1][Example]{\begin{trivlist}
\item[\hskip \labelsep {\bfseries #1}]}{\end{trivlist}}

\begin{document}
 \title{Asymptotic parameterizations in inverse limit spaces of dendrites}
 \author{Brent Hamilton}
 \address{Department of Mathematics, Baylor University, Waco, TX}
 \keywords{inverse limit, dendrite, unimodal, itinerary, critical point, arc-components}
 \subjclass{MSC (2010): 37B10, 54F15}
 \date{\today}

 \begin{abstract}
 In this paper, we study asymptotic behavior arising in inverse limit spaces of dendrites. In particular, the inverse limit is constructed with a single unimodal bonding map, for which points have unique itineraries and the critical point is periodic. Using symbolic dynamics, sufficient conditions for two rays in the inverse limit space to have asymptotic parameterizations are given. Being a topological invariant, the classification of asymptotic parameterizations would be a useful tool when determining if two spaces are homeomorphic.
 \end{abstract}
 \maketitle
\section{Introduction}
Inverse limit spaces with unimodal bonding maps on the interval have been extensively studied, and symbolic dynamics relating to the orbit of the critical point has proven fruitful in the study of such spaces \cite{BBBook, BrDi, BBSub, BAsymp, BEmb, RaStiSt, StSt}. A natural question to ask is when two inverse limit spaces are homeomorphic. Most notable along this line of inquiry is Ingram's Conjecture, which has generated much research \cite{BaDi, OIC, KPer, RaStiNR, SPreper}, culminating in the general solution in \cite{Ingram}. A natural extension of this work is to consider unimodal bonding maps on a space other than the interval. In \cite{B1} and \cite{B2}, a treatment of dendrites (i.e. continua that are locally connected and uniquely arcwise connected) was developed, wherein points are identified with their itineraries under a map $f$. This approach was used in \cite{B2} to prove a generalization of Ingram's Conjecture for $k$-stars. One advantage to this approach is that difficulties in determining \emph{admissibility}, that is, whether a given sequence is realized as the itinerary of a point in the space in question, largely vanish. In the interval case, determining admissibility can be quite laborious. 

As asymptotic rays are a topological invariant, their study can be useful in determining when two spaces are homeomorphic. Here, by a \emph{ray}, we mean the image of the nonnegative reals under a continuous bijection. Let $f:I\rightarrow I$ be a unimodal bonding map with critical point $t$, where $I$ is the unit interval, and let $(I,f)$ denote the corresponding inverse limit space. In \cite{BDH}, it was shown that if $t$ is periodic with period $N$, then $(I,f)$ contains at most $2(N-2)$ asymptotic rays, excepting any on the $N$ arc-components with endpoints. This work was extended in \cite{BAsymp}, wherein a symbolic characterization of asymptotic rays is presented. Additionally, ignoring admissibility, sufficient conditions for rays to be asymptotic were given, and all resultant asymptotic rays, up to period 8, are listed. In some cases, the upper bound of $2(N-2)$ is witnessed. 

In this paper, we take many of the results given in \cite{BAsymp}, and give analagous results for inverse limit spaces of dendrites, as developed in \cite{B1} and \cite{B2}. The sufficient conditions for arc-components to have asymptotic parameterizations given in \cite{BAsymp} carry over to the dendrite case, with some surprising differences. For one, all the itineraries are easily seen to be admissible. This gives rise to new itineraries which were not admissible in the unimodal case. Additionally, the space $(I,f)$ is known to be chainable, whereas the inverse limit space of dendrites is not. This ultimately results in a countably infinite collection of distinct asymptotic rays on a single arc-component. 

The organization of this paper is as follows. The next section gives a survey of the relevant notation and results from \cite{B1} and \cite{B2}. Section 3 extends many results from \cite{BAsymp} to the case of dendrites. Section 4 states the main theory, and lists the relevant asymptotic arc-components for all admissible kneading sequences up to period 5. 

\section{Symbolic Representation of Dendrites}
We use the symbols $\mathbb{N}$, $\mathbb{Z}$, and $\mathbb{Z_-}$ to denote, respectively, the positive integers, the integers, and the negative integers. 
Suppose $X$ is a topological space, and $f:X\rightarrow X$ is a continous map. We define $(X,f)$ to be a subset of the product space $X^{\mathbb{Z_-}}$, where $(X,f)=\{(\ldots, x_{-2}, x_{-1}): f(x_{i-1})=x_{i}\text{ for all }i\in\mathbb{Z_-}\}$. We say that $(X,f)$ is the \emph{inverse limit} of $X$ with \emph{bonding map} $f$. 

A \emph{dendrite} is a locally connected, uniquely arcwise connected continuum. If $D$ is a dendrite, a continuous function $f:D\rightarrow D$ is said to be \emph{unimodal} if it is locally one-to-one, excepting a single point $t$, referred to as the \emph{critical point}. A \emph{pseudoleg} is a union of components of $D_{\tau}-\{t\}$ on which $f$ is one-to-one. For the entirety of this paper, we shall assume the critical point $t$ is periodic, and that $D_{\tau}-\{t\}$ consists of two pseudolegs. The pseudoleg containing $f(t)$ is labelled $L_1$, and the other pseudoleg is labelled $L_2$. If $x\in D$, we define the \emph{itinerary} of $x$, denoted $\imath(x)=\imath_0\imath_1\imath_2\ldots$ as follows:

 \begin{displaymath}
   \imath_n = \left\{
     \begin{array}{lr}
       1 \text{ if } f^n(x)\in L_1\\
       2 \text{ if } f^n(x)\in L_2\\
       * \text{ if } f^n(x)=t
     \end{array}
   \right.
\end{displaymath}

By the \emph{kneading sequence}, we mean the itinerary of the critical point $t$. If $\imath(x)\neq\imath(y)$ whenever $x\neq y$, then we say $f$ has the \emph{unique itinerary property}. A unimodal map satisfying the unique itinerary property is said to be \emph{tentish}. We topologize the symbol set $\{*,1,2\}$ with the basis $\{\{1\},\{2\},\{*,1,2\}\}$, and extend this to the product topology for $\{*,1,2\}^{\mathbb{N}}$ and $\{*,1,2\}^{\mathbb{Z}}$. If $a$,$b\in\{*,1,2\}$, we say $a\approx b$ if either $a=b$ or if at least one of $a$ or $b$ equals $*$, and we expand the definition of ``$\approx$" to product spaces in the obvious way. Let $\sigma$ denote the shift map on $\{*,1,2\}^{\mathbb{N}}$. A sequence $\tau=\tau_0\tau_1\tau_2\ldots\in\{*,1,2\}^{\mathbb{N}}$ is said to be \emph{acceptable} if, for all $n\in\mathbb{N}$, $\sigma^n(\tau)=\tau$ whenever $\sigma^n(\tau)\approx\tau$. Given an acceptable sequence $\tau$, an element $x=x_0x_1x_2\ldots\in\{*,1,2\}^{\mathbb{N}}$ is said to be \emph{$\tau$-admissible} if, whenever $x_n=*$, we have $\sigma^{n}x=\tau$ and $\sigma^{n}x=\tau$ whenever $\sigma^{n}(x)\approx\tau$. The space $D_{\tau}$ is defined to be the set of all $\tau$-admissible sequences, and was shown in \cite{B1} to be a dendrite. 

\begin{proposition}\cite[2.17]{B1}
Let $f:D\rightarrow D$ be a tentish dendrite map with kneading sequence $\tau$. Then the itinerary map $\imath:D\rightarrow D_{\tau}$ is a homeomorphism onto its range, and $\imath$ is a topological conjugation between $f$ and $\sigma |_{\imath(D)}$. 
\end{proposition}

Let $\hat{x}=\ldots \hat{x}_{-2}\hat{x}_{-1}.\hat{x}_0\hat{x}_1\ldots\in\{*,1,2\}^{\mathbb{Z}}$. For each $n\in\mathbb{Z}$, define $\pi_{n}(\hat{x})=\hat{x}_n\hat{x}_{n+1}\hat{x}_{n+2}\ldots$. Let $\hat{D}_{\tau}=\{\hat{x}\in\{*,1,2\}^{\mathbb{Z}}:\pi_n(\hat{x})\in D_{\tau}\}$, and let $\hat{\sigma}$ denote the shift map on $\hat{D}_{\tau}$. 

\begin{proposition}\cite[2.4]{B2}
Define $h:\hat{D}_{\tau}\rightarrow (D_{\tau},\sigma)$ by $h(\hat{x})=(\ldots,\pi_{-3}(\hat{x}),\pi_{-2}(\hat{x}),\pi_{-1}(\hat{x}))$. Then $h$ is a topological conjugation between $\hat{\sigma}$ and the corresponding shift map for $(D_{\tau},\sigma)$. 
\end{proposition}

 \begin{proposition}\label{nolargefolds}
 \cite[2.25]{B1}
Let $A$ be an arc in $\hat{D}_{\tau}$ with endpoints $\hat{x}$ and $\hat{y}$, and  suppose $k=\min\{i|x_i\not\approx y_i\}$ is finite. Then if $\hat{z}\in A$ and $i<k$, we have $x_i\approx z_i\approx y_i$.
 \end{proposition}

 If $S\subseteq D_{\tau}$, we let $[S]$ denote the smallest connected subset of $D_{\tau}$ containing $S$. We may use $[x,y]$, in place of $[\{x,y\}]$,  to denote the unique arc in $D_{\tau}$ having $x$ and $y$ as endpoints, and $(x,y)=[x,y]-\{x,y\}$. Given two points $x$ and $y$ of $D_{\tau}$, it is often useful to find a point in $(x,y)$. The following technique, dubbed the ``$\mu$-process," was developed in \cite{B1}, and was useful in proving many results (e.g. that $D_{\tau}$ is connected). If $x=x_0x_1\ldots$ and $y=y_0y_1\ldots$ are distinct elements of $D_{\tau}$, then there exists a minimal $n$ so that $x_n\not\approx y_n$. Define $\mu'(x,y)=\mu'_1\mu'_2\ldots$ as follows. If $i<n$, let $\mu'_i\in\{x_i,y_i\}$ be chosen so that $\mu'_i\neq *$. If $i=n$, set $\mu'_i=*$, and for $i>n$, $\mu'_i=\tau_{i-n}$. Then there exists a unique $\mu\in D_{\tau}$ with $\mu\approx\mu'(x,y)$. Moreover, $\mu\in (x,y)$. 

We conclude this section with a characterization of arc-components in $\hat{D}_{\tau}$. If $\hat{x}=\ldots \hat{x}_{-2}\hat{x}_{-1}.\hat{x}_{0}\hat{x}_{1}\ldots$ is an element of $\hat{D}_{\tau}$, we define the \emph{backwards itinerary} of $\hat{x}$, denoted $e(\hat{x})$, to be the reverse sequence $\ldots \hat{x}_{-3}\hat{x}_{-2}\hat{x}_{-1}$. We define the equivalence class $e_*$ via $e_*=\{e(\hat{x}):\hat{x}\in\hat{D}_{\tau}\text{ and } e(\hat{x})_i=e_i \text{ whenever } i\leq M \text{ for some } M\in\mathbb{Z}_-\}$. Given two backwards itineraries $e=e(\hat{x})$ and $\widetilde{e}=e(\hat{y})$, we define the \emph{sequence of discrepancies} as follows.  Let $k_1 = \min\{i|e_{-i}\neq \widetilde{e}_{-i}\}$ and inductively define $k_{i+1}=\min\{i>k_i|e_{-i}\neq\widetilde{e}_{-i}\}$. We call $k_1$ the \emph{first discrepancy} between $e$ and $\widetilde{e}$. If $k_{i+1}$ does not exist, we leave it undefined and say the sequence of discrepancies is finite. In this case, the sequences $e$ and $\widetilde{e}$ have the same tails, and we have $e_*=\widetilde{e}_*$.

\begin{proposition}\label{arccharacterization}
\cite[2.7]{B2}
Let $\hat{x}$ and $\hat{y}$ be points in $\hat{D}_{\tau}$, where $\tau$ is of period $N$. Let $\{k_i\}$ denote the sequence of discrepancies between $e=e(\hat{x})$ and $\widetilde{e}=e(\hat{y})$. Then $\hat{x}$ and $\hat{y}$ are in the same arc-component if and only if $\{k_i\}$ is finite or if there exists a natural number $M$ so that if $k_i$,$k_j \geq M$, we have $k_i\equiv k_j \mod N$ and, for each $i$, $x_{-k_{i+1}}\ldots x_{-k_i+1} \approx (\tau_{0}\ldots\tau_{N-1})^{n_i} \approx y_{-k_{i+1}}\ldots y_{-k_i+1}$, where $n_i=(k_{i+1}-k_{i})/N$.
\end{proposition}


 
\section{Asymptotic Rays in Inverse Limits of Dendrites}

  Suppose $\hat{x}\in\hat{D}_{\tau}$. For each $N\in\mathbb{Z}$, let $T_{N}(\hat{x})=\{\hat{y}\in\hat{D}_{\tau}:y_i\approx x_i \text{ for all } i\leq N\}$, and let $T(\hat{x})=T_{-1}(\hat{x})$. If $e=e(\hat{x})$, we may write $T_N(e)$ in place of $T_N(\hat{x})$. It is easily verified that $T_N(\hat{x})$ is a closed subset of $\hat{D}_{\tau}$, and is homeomorphic to $\pi_N(T_N(\hat{x}))$.
  
 \begin{lemma}\cite[2.14]{B1}
Suppose $y\in D_{\tau}$ and $e$ is a backwards itinerary for some point in $\hat{D}_{\tau}$. Then there exists a unique $\hat{x}\in\hat{D}_{\tau}$ such that $\pi_0(\hat{x})= y$ and $\hat{x}\in T(e)$.
\end{lemma}

 \begin{lemma}\label{connected}
 $T_{N}(e)$ is uniquely arc-wise connected.
 \end{lemma}
 \begin{proof}
 If $e_i=*$ for some $i\leq N$, then $T_N(e)$ consists of a single point. Suppose $e_i\neq *$ for all $i\leq N$. Let $\hat{x}$ and $\hat{y}$ be distinct elements of $T_N(e)$. Since $D_{\tau}$ is arc-wise connected, there exists an arc $A$ in $D_{\tau}$ having $\pi_N(\hat{x})$ and $\pi_N(\hat{y})$ as endpoints. Let $k=\min\{i>N:x_i\neq y_i\}$ and let $z\in A$. By admissibility of $\pi_N(\hat{x})$ and $\pi_N(\hat{y})$, $k$ is finite. By Proposition \ref{nolargefolds}, if $N\leq i < k$, we have $x_i\approx z_{i-N+1} \approx y_i$. For $i\geq N$, let $p'_i=z_{i-N+1}$, and for $i<N$, let $p'_i=e_i$. By the previous lemma, there exists a unique $\hat{p}\in\hat{D}_{\tau}$ with $\hat{p}\approx p'$. Moreover, $\pi_N(\hat{p}) =z$. Hence, we have $\pi_N(T(e)) \supseteq A$. Since $\pi_N$ is a homeomorphism from $T_N(e)$ onto its image, $\pi_N^{-1}(A)$ is a subarc of $T_N(e)$, with endpoints $\hat{x}$ and $\hat{y}$. Furthermore, since $D_{\tau}$ is uniquely arc-wise connected, $T_N(e)$ is as well. 
\end{proof}
 
\begin{proposition}
 $T_N(e)$ is a subcontinuum of $\hat{D}_{\tau}$.
 \end{proposition}
 \begin{proof}
 Since $T_N(e)$ is a closed subset of the compact space $\hat{D}_{\tau}$, we have that $T_N(e)$ is compact. This, coupled with the previous proposition, yields the desired result.
 \end{proof}
 
If $e$ and $\widetilde{e}$ are backwards itineraries occurring in $\hat{D}_{\tau}$, we seek a way to determine when $T(e)$ and $T(\widetilde{e})$ are share a common boundary point. To that end, we make the following definition: for each $0\leq i < N$, let $\beta^{i}(e)=$ $\max\{k: e_{-k}e_{-(k-1)}\ldots e_{-1}\approx \tau_{0}\tau_{1}\ldots \tau_{k-1}\text{ and } k \equiv i \mod N\}$. If no such match exists, we will leave $\beta^{i}(e)$ undefined. 

\begin{proposition}
Suppose $\tau$ is of period $N$, and $e$ is a backwards itinerary with $e_{-k}=*$ for some positive integer $k$. Then, there exists a unique $0\leq i < N$ such that $\beta^{i}(e)$ is defined. In particular, $i\equiv k\mod N$.
\end{proposition}
\begin{proof}
This follows easily from the definition of admissiblity.
\end{proof}

 \begin{proposition}\label{adjacency}
 Suppose $e$ is a backwards itinerary, with $e_i\neq *$ for all $i$ and $\tau$ is of period $N$. Suppose $\beta^{k}(e)$ is defined. Define $\widetilde{e}$ by setting $\widetilde{e}_i$ to be either 1 or 2 when $\beta^{k} \leq i < 0 \text{ and } i\equiv k \mod N$, and $\widetilde{e}_i=e_i$ otherwise. If $e\neq \widetilde{e}$, then $T(e)\cap T(\widetilde{e})$ consists of a single point.
 \end{proposition}
 \begin{proof}
 Suppose $\beta=\beta^k(e)$ is finite. Define $\hat{p}=\ldots p_{-2}p_{-1}p_0p_1\ldots$ by setting

 \begin{displaymath}
  {p}_i = \left\{
     \begin{array}{lr}
       e_i \text{ if } i < -\beta \\
       \tau_{k+i} \text{ if } i\geq -\beta
     \end{array}
   \right.
\end{displaymath}
 Note that $\pi_{-\beta}(\hat{p})=\tau$. Hence, if $i\geq -\beta$, $\pi_{i}(\hat{p})\in D_{\tau}$. Moreover, if $i<-\beta$, then $\pi_{i}(\hat{p})\not\approx\tau $, since $k$ was chosen maximally modulo $N$. Thus, $\hat{p}\in \hat{D}_{\tau}$.

Now, we show $\hat{p}\in T(e)$. Let $O=\ldots O_{-1}\times O_0\times O_1 \ldots$ be a basic open set containing $\hat{p}$. Pick $n\in\mathbb{N}$ so that whenever $|i|>n$ we have $O_i=\{*,1,2\}$. Define $\hat{x}$ by setting

\begin{displaymath}
   \hat{x}_i = \left\{
     \begin{array}{lr}
       e_i \text{ if } i < 0\\
       p_i \text{ if } 0\leq i < n \text{ and } p_i\neq *\\
       2 \text{ otherwise }
     \end{array}
   \right.
\end{displaymath}

Then we clearly have $\hat{x}\in O \cap T(e)$. Since $T(e)$ is closed, it follows that $\hat{p}\in T(e)$. That $\hat{p}\in T(\widetilde{e})$ follows similarly.

It remains to show that $T(e)\cap T(\widetilde{e})\subseteq\{\hat{p}\}$. Suppose $\hat{q}\in T(e)\cap T(\widetilde{e})$. Pick $j$ minimally so that $e_{j}\neq \widetilde{e}_j$, and note that $j\equiv \beta\mod N$. If $\hat{q}_j\neq *$, we may construct an open set containing $\hat{q}$, and missing one of either $T(e)$ or $T(\widetilde{e})$. Hence, $\hat{q}_j=*$. If $\hat{q}_{-\beta}\neq *$, then $\pi_{-\beta}(\hat{q})\approx \tau$, but $\pi_{-\beta}(\hat{q})\neq \tau$, contradicting admissibility. Hence, $\hat{q}_{-\beta}= *$, which implies $\hat{q}=\hat{p}$. 

The proof for when $\beta=\infty$ is similar.
\end{proof}

\begin{corollary}\label{project}
Suppose $\tau$ is of period $N$, and let $e$ be a backwards itinerary occurring in $\hat{D}_{\tau}$. Then \bwip $\supseteq [\{\sigma^i(\tau)|\beta^{i}(e) \text{ is defined}\}]$.
\end{corollary}
\begin{proof}
By Lemma \ref{adjacency}, whenever $\beta^{i}(e)$ is defined, it corresponds to a boundary point of $T(e)$ which projects to $\sigma^i(\tau)$. \bwip is connected, and $[\{\sigma^i(\tau)|\beta^{i} \text{ is defined}\}]$ is the smallest connected subset of $D_{\tau}$ containing the $\sigma^i(\tau)$'s. 
\end{proof}

\begin{example}
Let $\tau = \overline{*112}$ and $e=1^{\infty}$. Then the point $\hat{p}_1$ with backwards itinerary $1^{\infty}*$ is a boundary point for $T(e)$, and is adjacent to the continuum $T(e_1)$, where $e_1=1^{\infty}2$. Moreover, $T(e)$ also shares boundaries with $T(e_2)$ and $T(e_3)$, where $e_2=1^{\infty}21$ and $e_3=1^{\infty}211$. Hence, it follows that $T(e)$ contains a branch point (i.e., a point $\hat{p}$ so that $T(e)-\{\hat{p}\}$ consists of more than two components). This corresponds to the \emph{central branching point} of $D$: $1^{\infty}$(cf \cite{B1}, Theorem 1.22 and Definition 1.23).

In general, a boundary point of $T(e)$ may also be a branch point. For example, again let $\tau = \overline{*112}$, and define $\hat{p}$ to be the point with backwards itinerary $1^{\infty}*112*112$. Then $\hat{p} \in T(e_1) \cap T(e_2) \cap T(e_3)$, where  $e_1=1^{\infty}21112$, $e_2=1^{\infty}21121112$, and $e_3=1^{\infty}22112$. And, if $\beta^k(e)=\infty$ for some $k$, there exists an infinite collection of backwards itineraries, whose corresponding continua share a common boundary point. For example, for each $n$, let $e^n=(1112)^{\infty}(2112)^n$. Then $(*112)^{\infty}.(*112)^{\infty}\in \cap_{n=1}^{\infty}T(e^n)$.  

\end{example}

\begin{lemma}\label{threebwi}
Suppose $A=[\hat{x},\hat{y}]$ is an arc in $\hat{D}_{\tau}$, where $\tau$ is of period $N$. If $\{\hat{z}^i\}$ is a sequence of points of $A$ converging to $\hat{y}$, then there exists an integer $M$ so that if $i,j\geq M$, we have $e(\hat{z}^i)=e(\hat{z}^j)$. 
\end{lemma}
\begin{proof}
Suppose $e(\hat{z}^i)\neq e(\hat{z}^j)$ infinitely often. Then, between each such pairing, there exists a $\hat{p}^i$ with $\hat{p}^i_{-k(i)}=*$ for some $k(i)\in\mathbb{N}$. By passing to a subsequence if necessary, we may assume each $k(i)$ is congruent modulo $N$. It follows that $\hat{y}$ is a shift of $(\tau_0\tau_1\ldots\tau_{N-1})^{\infty}.(\tau_0\tau_1\ldots\tau_{N-1})^{\infty}$, as otherwise we may construct an open set containing $\hat{y}$ and at most finitely many of the $\hat{p}^i$'s. By Proposition \ref{arccharacterization}, $e_*(\hat{p}^i)\approx e_*(\hat{y})$ for each $i$. Fix $i_0$, and pick $M$ so that whenever $j>M$, we have $\hat{p}^{i_0}_{-j}\approx\hat{y}_{-j}$. By Proposition \ref{nolargefolds}, whenever $j>M$ and $i\geq i_0$, we have $\hat{p}^{i}_{-j}\approx\hat{y}_{-j}$. This leaves only finitely many options for $e(\hat{p}^i)$ when $i\geq i_0$. Hence, there exists an $i_1$ so that whenever $i,j\geq i_1$, we have $e(\hat{p}^i)=e(\hat{p}^j)$. By admissibility, this implies $\hat{p}^i=\hat{p}^j$ whenever $i,j\geq i_1$, providing a contradiction.
\end{proof}

\begin{proposition}\label{arcbwi}
Suppose $A=[\hat{x},\hat{y}]$ is an arc in $\hat{D}_{\tau}$, where $\tau$ is of period $N$. Then there exists finitely many backwards itineraries occuring on $A$. 
\end{proposition}
\begin{proof}
If $\hat{z}\in [\hat{x},\hat{y}]$ with $\hat{z}_{-i}\neq *$ for all $i\in\mathbb{N}$, then the set $\{\hat{p}\in [\hat{x},\hat{y}]:e(\hat{p})=e(\hat{z})\}$ is open in the subspace $[\hat{x},\hat{y}]$. If $\hat{z}\in [\hat{x},\hat{y}]$ with $\hat{z}_{-i}= *$ for some $i$, then, by Lemma \ref{threebwi} we may find an open set (in the topology of $[\hat{x},\hat{y}]$) containing $\hat{z}$ and at most three backwards itineraries. This gives us an open cover of the compact space $[\hat{x},\hat{y}]$. Taking a finite subcover concludes the proof. 
\end{proof}

Suppose $\phi:[0,\infty)\rightarrow\hat{D}_{\tau}$ is a continous bijection. We call the image of $\phi$ a \emph{ray}, and $\phi$ a \emph{parameterization}. Suppose $\phi :[0,\infty)\rightarrow\inversedendrite$ parameterizes a ray $\ray$. Let $e=e(\phi(0))$, and suppose that $e_{-i}\neq *$ for all $i\leq -1$. As $s$ increases, the backwards itineraries $e(\phi(s))$ may also change. Pick $s_1$ minimal so that $e(\phi(s_1))\neq e$. Then, by Proposition \ref{arcbwi}, $\phi(s_1)$ is a boundary point between $Te$ and $Te(\phi(s_1+\varepsilon_1))$ for sufficiently small $\varepsilon_1$. Let $Re=e(\phi(s_1+\varepsilon_1))$. Continue inductively, picking $s_n>s_{n-1}$ minimally with $e(\phi(s_n))\neq R^{n-1}e$. Then $\phi(s_n)$ is a boundary point between $TR^{n-1}e$ and $Te(\phi(s_n+\varepsilon_n))$ for sufficiently small $\varepsilon_n$, and let $R^ne=e(\phi(t_n+\varepsilon_n))$. We define the \emph{folding pattern}, $\{\alpha_n(\ray)\}$, or simply \fp when the ray $\ray$ is understood, by letting $\alpha_n$ be the sequence of discrepancies between $R^{n-1}e$ and $R^ne$. If $\ray'$ is a ray originating in $T(\widetilde{e})$ with folding pattern \fpti we let $d_n(\ray,\ray')$ denote the first discrepancy between $R^ne$ and $R^n\widetilde{e}$. Note that each $\alpha_n$ is a (potentially finite) sequence, each element of which is congruent modulo $N$. We let $C(\alpha_n)$ denote the least nonnegative element of this congruence class. Let $\phi'$ parameterize $\ray'$, and let $d$ be a metric compatible with the topology of $\hat{D}_{\tau}$. We say the rays $\ray$ and $\ray'$ are \emph{asymptotic} provided that $d(\phi(s),\phi'(s))\rightarrow 0$. Of course, this definition allows for trivial asymptotic behavior (e.g. any ray is asymptotic to itself). However, our chief concern will be with asymptotic rays that are on distinct arc-components.  

There are notable differences in our use of $R^ne$, as compared with the development in \cite{BAsymp, BEmb}, for the space $(I,f)$. In particular, in $(I,f)$, $R^ne$ is a function of the backwards itinerary $e$. Whereas in our treatement, given a backwards itinerary $e$, there are multiple valid choices for $Re$, depending on which of the $\beta^{i}e$'s are defined. Indeed, if $\beta^{i}e=\infty$ for some $i$, there are infinitely many options for $Re$. Additionally, given a backwards itinerary $e$, once choices have been assigned to $R^ne$ for each $n\in\mathbb{N}$, this defines a unique ray in $\hat{D}_{\tau}$. We state this more formally with the following proposition.


\begin{proposition}
Suppose $e=e(\hat{x})$ for some $\hat{x}\in\hat{D}_{\tau}$. Let $e^1$ be a backwards itinerary, distinct from $e$, so that $T(e^1)$ shares a common boundary point with $T(e)$. Let $\gamma_1$ denote the sequence of discrepancies between $e$ and $e^1$. Continue inductively, letting $e^n\not\in\{e^{i}:i<N\}$ be chosen so that $T(e^n)$ and $T(e^{n-1})$ share a common boundary point, and let $\gamma_n$ denote the sequence of discrepancies between $e^{n-1}$ and $e^n$. Then there exists a unique ray $\ray$ in $\hat{D}_{\tau}$, so that $\alpha_n(\ray)=\gamma_n$ for all $n\in\mathbb{N}$.
\end{proposition}
\begin{proof}
 This follows from Lemma \ref{connected} and Proposition \ref{adjacency}.
\end{proof}

  
\begin{example}
Suppose $\tau=\overline{*112}$, and let $e=\overline{1112}$, $e^1=\overline{2112}$, $e^2=\overline{2112} 2111$, and $e^3=\overline{2112} 2211$. Then, by Proposition \ref{adjacency}, there exists a ray $\ray$ beginning in $Te$, and travelling through $Te^3$ by way of $Te^1$ and $Te^2$. For such a ray, we have $e^1=Re$, $e^2=R^2e$, $e^3=R^3e$, and the folding pattern for $\ray$ begins with $\alpha_1=\{4n\}_{n=1}^{\infty}$, $\alpha_2=\{1\}$, and $\alpha_3=\{3\}$. 

As mentioned previously, there are several (in fact, infinitely many) distinct rays originating in $Te$. Indeed, we could take $\alpha_1$ to be any subsequence (finite or infinite) of $\{4n\}_{n=1}^{\infty}$, which corresponds to a distinct path along the arc-component for which $Te$ is a subset. 

\end{example}

\begin{proposition}
$\pi_0(T(R^ne)\cap \ray)= [\sigma^{C(\alpha_{n-1})}(\tau), \sigma^{C(\alpha_{n})}(\tau)]$
\end{proposition}
\begin{proof}
This follows from Proposition \ref{project}.
\end{proof}

The following can be stated more generally (cf. \cite{BAsymp}, Proposition 1); but the following will suit our purposes.

\begin{proposition}\label{symasymp}
Suppose $\rho$ is a metric compatible with the topology of ${D}_{\tau}$, and $\tau$ is of period $N$. Let $\ray$ and $\ray '$ be rays in $\hat{D}_{\tau}$ with respective folding patterns \fp and \fpti. If $\rho(\sigma^{C(\alpha_n)}\tau,\sigma^{C(\widetilde{\alpha}_n)}\tau)\rightarrow 0$ and $d_n(\ray,\ray')\rightarrow\infty$, then $\ray$ and $\ray'$ are asymptotic. 
\end{proposition}
\begin{proof}
Without loss of generality, suppose $\rho(\sigma^{C(\alpha_n)}\tau,\sigma^{C(\widetilde{\alpha}_n)}\tau)=0$ for all $n$. Let $e=e(\phi(0))$ and $e'=e(\phi'(0))$. Then $\pi_0(R^ne\cap \ray)=\pi_0(R^ne'\cap \ray')$. For each $n$, let $\phi_n:[n,n+1]\rightarrow (R^ne\cap \ray)$ be a parameterization of $(R^ne\cap \ray)$. Similarly, let $\phi'_n:[n,n+1]\rightarrow (R^ne'\cap \ray')$ parameterize $(R^ne'\cap \ray')$ so that $\pi_0(\phi(t))=\pi_0(\phi'(t))$, and expand these in the obvious way to get the parameterizations $\phi$ and $\phi'$. The condition $d_n(\ray,\ray')\rightarrow\infty$ implies that $d(\phi(t),\phi'(t))\rightarrow 0$, where $d$ is a metric for $\hat{D}_{\tau}$. 
\end{proof}

\section{Main Results}

\begin{lemma}\label{nstep}
Suppose $e$ and $\widetilde{e}$ are backwards itineraries, and let $j \leq \underline{d}$, where $\underline{d}$ is the first discrepancy between $e$ and $\widetilde{e}$.  Let $w_1w_2\ldots w_j \in \{1,2\}^j$, $\psi=\ldots e_{-(j+2)}e_{-(j+1)}w_1w_2\ldots w_j$, and $\widetilde{\psi}=\ldots \widetilde{e}_{-(j+2)}\widetilde{e}_{-(j+1)}w_1w_2\ldots w_j$. Then there exists an $n$ so that $\psi = R^ne$ and $\widetilde{\psi}=R^n\widetilde{e}$.
\end{lemma}
\begin{proof}
By Proposition \ref{arccharacterization}, there exists rays $R$ and $\widetilde{R}$, respectively originating in $T(e)$ and $T(\widetilde{e})$ and peregrinating through $T(\psi)$ and $T(\widetilde{\psi})$. Hence, there exists integers $n_1$ and $n_2$ so that $R^{n_1}e=\psi$ and $R^{n_2}\widetilde{e}=\widetilde{\psi}$. Hence, we need only show $n_1=n_2$. 

Suppose $n_1=1$. Let $\{k_i\}_{i=1}^{m}$ be the sequence of discrepancies between $e$ and $\psi$.  Then $e_{-k_m}\ldots e_{-1}\approx \tau_0\ldots\tau_{k_m-1}$. Since $k_m\leq j$, we also have $\widetilde{e}_{-k_m}\ldots\widetilde{e}_{-1}\approx\tau_0\ldots\tau_{k_m-1}$. Moreover, by Proposition \ref{adjacency}, all the $k_i$'s are congruent modulo the period of the kneading sequence. Hence, $T(\widetilde{e})$ and $T(\widetilde{\psi)}$ share a common boundary point, and $n_2=1$. Now, suppose $n_1>1$ and we proceed by induction. By Proposition \ref{nolargefolds}, the first discrepancy between $R^{n_1-1}e$ and $e$ is at most $k$. By the inductive hypothesis, the last discrepancy between $R^{n_1-1}e$ and $R^{n_1-1}\widetilde{e}$ is at most $k$. Apply the same argument as used in the base case to conclude the proof.
\end{proof}

Analogues of the following theorems were given in \cite{BAsymp} for inverse limits of unimodal maps of the interval. Before proceeding, we will introduce some new notation. If $x=x_0x_1x_2\ldots$ and $y=y_0y_1y_2\ldots$ are elements of $D_{\tau}$, let $\underline{d}(x,y)=\min\{i:x_i\not\approx y_i\}$. Additionally, if $a\in\{1,2\}$, we use $a'$ to denote the unique element of $\{1,2\}-\{a\}$.  

\begin{theorem}\label{case1}
Suppose $\tau=\overline{*12^{N-2}}$, where $\tau_0=*$. Let $l\in\{1,2,\ldots,N-2\}$. If
\begin{align*}
		e&=\overline{12^{N-3}1}12^{N-2}1\text{ and}\\ 
		\widetilde{e}&=\overline{12^{N-3}1}12^{N-2}12^{l}\text{,}
\end{align*}
 then there exists asymptotic rays originating in $T(e)$ and $T(\widetilde{e})$.
\end{theorem}

\begin{proof}
Let
\begin{align*}
       Re&=\overline{12^{N-3}1}12^{N-2}12^{N-2}1 \\
        R\widetilde{e}&=\overline{12^{N-3}1}12^{N-2}12^{l-1}1
\end{align*}
By Lemma \ref{nstep}, there exists an integer $n$ so that we may define
\begin{align*}
       R^{n}e&=\overline{12^{N-3}1}12^{N-2}12^{N-l-1}12^{l-1} \\
        R^{n}\widetilde{e}&=\overline{12^{N-3}1}12^{N-2}112^{l-1}
\end{align*}
 Next, take $\alpha_{n+1}(e)=l+1$ and $\alpha_{n+1}(\widetilde{e})=N+l+1$, and we have
\begin{align*}
       R^{n+1}e&=\overline{12^{N-3}1}12^{N-2}12^{N-l-2}112^{l-1} \\
                &= e2^{N-l-2}112^{l-1}\\
        R^{n+1}\widetilde{e}&=\overline{12^{N-3}1}12^{N-2}12^{N-2}112^{l-1}\\
                &=\widetilde{e}2^{N-l-2}112^{l-1}
\end{align*}
Hence, each of $R^{n+1}e$ and $R^{n+1}\widetilde{e}$ are left-shifts of the original itineraries $e$ and $\widetilde{e}$, and all folds thus far have been congruent modulo $N$. We now proceed by induction. Suppose $n_{i-1}$ is defined so that $R^{n_{i-1}}e=eA$ and $R^{n_{i-1}}\widetilde{e}=\widetilde{e}A$, and all folds thus far have been congruent modulo $N$. Pick $m$ so that

\begin{align*}
       R^{m}e&=\overline{12^{N-3}1}12^{N-2}1V_{|A|}\\
        R^{m}\widetilde{e}&=\overline{12^{N-3}1}12^{N-2}12^lV_{|A|}
\end{align*}
where, $V_{n}$ is a \{1,2\}-block with $V_{n}\approx \tau_1\ldots\tau_{n}$. Next, take $\alpha_{m+1}(e)=|A|+N+1$ and $\alpha_{m+1}(\widetilde{e})=|A|+1$ and obtain
\begin{align*}
       R^{m+1}e&=\overline{12^{N-3}1}12^{N-2}12^{N-2}1V_{|A|}\\
        R^{m+1}\widetilde{e}&=\overline{12^{N-3}1}12^{N-2}12^{l-1}1V_{|A|}
\end{align*}

Define $n_i-1$ so that

\begin{align*}
       R^{n_i-1}e&=\overline{12^{N-3}1}12^{N-2}12^{N-l-1}V_{|A|+l}\\
        R^{n_i-1}\widetilde{e}&=\overline{12^{N-3}1}12^{N-2}1V_{|A|+l}
\end{align*}
Take $\alpha_{n_i}(e)=|A|+l+1$ and $\alpha_{n_i}(\widetilde{e})=|A|+l+N+1$, and we have
\begin{align*}
 R^{n_i}e&=\overline{12^{N-3}1}12^{N-2}12^{N-l-2}1V_{|A|+l}\\
        &=e2^{N-l-2}1V_{|A|+l}\\
 R^{n_i}\widetilde{e}&=\overline{12^{N-3}1}12^{N-2}12^{N-2}1V_{|A|+l}\\
        &=\widetilde{e}2^{N-l-2}1V_{|A|+l}
\end{align*}

and all folds have been congruent modulo $N$. Hence, we have constructed rays $\ray$ and $\ray'$ for which $d_n(\ray,\ray')\rightarrow \infty$, and $\pi_0(R^ne\cap \ray)=\pi_0(R^n\widetilde{e}\cap\ray')$ for all $n$. Proposition \ref{symasymp} implies the constructed rays are asymptotic.
\end{proof}

\begin{theorem}\label{case2}
Suppose there exists a $N/2<k<N$ such that $\underline{d}(\sigma^{k}\tau,\tau)\geq N-k$, and let $\nu_1\ldots\nu_N \approx \tau_1\ldots\tau_N$, where $\nu_N \in \{1,2\}$ is chosen so that $\underline{d}(\sigma^{k}\nu,\tau)>N-k$. If
\begin{align*}
		e&=\overline{\nu_1\ldots\nu_k'}\nu_1\ldots\nu_{N} \text{ and}\\
		\widetilde{e}&=\overline{\nu_1\ldots\nu_{N-k}'\nu_{N-k+1}\ldots\nu_{k}}\nu_1\ldots\nu_{N-k}'\nu_{N-k+1}\ldots\nu_N'\text{,}
\end{align*}
then there exists asymptotic rays originating in $T(e)$ and $T(\widetilde{e})$. 
\end{theorem}
\begin{proof}
We begin by taking $\alpha_1(e)=N+1$ and $\alpha_1(\widetilde{e})=1$, and obtain

\begin{align*}
Re &= \overline{\nu_1\ldots\nu_k'}\nu_1\ldots\nu_k\nu_1\ldots\nu_N\\
R\widetilde{e} &= \overline{\nu_1\ldots\nu_{N-k}'\nu_{N-k+1}\ldots\nu_k}\nu_1\ldots\nu_{N-k}'\nu_{N-k+1}\ldots\nu_N
\end{align*}
The hypothetical condition $\underline{d}(\sigma^k\nu,\tau)>N-k$ implies $\nu_1\ldots\nu_k\nu_1\ldots\nu_{N-k}=\nu_1\ldots\nu_N$. Additionally, since the first discrepancy between $e$ and $\widetilde{e}$ is not less than $k$, we may apply Proposition \ref{nolargefolds} and pick $n$ so that
\begin{align*}
R^ne &= \overline{\nu_1\ldots\nu_k'}\nu_1\ldots\nu_k\nu_1\ldots\nu_{N-k}\nu_1\ldots\nu_k\\
     &= \overline{\nu_1\ldots\nu_k'}\nu_1\ldots\nu_N\nu_1\ldots\nu_k\\
R^n\widetilde{e} &= \overline{\nu_1\ldots\nu_{N-k}'\nu_{N-k+1}\ldots\nu_k}\nu_1\ldots\nu_{N-k}'\nu_{1}\ldots\nu_k
\end{align*}
Next, take $\alpha_{n+1}(e)=N+k+1$ and $\alpha_{n+1}(\widetilde{e})=k+1$, and we have
\begin{align*}
R^{n+1}e &= \overline{\nu_1\ldots\nu_k'}\nu_1\ldots\nu_k\nu_1\ldots\nu_N\nu_1\ldots\nu_k\\
         &= \overline{\nu_1\ldots\nu_k'}\nu_1\ldots\nu_N\nu_{N-k+1}\ldots\nu_N\nu_1\ldots\nu_k\\
         &= e\nu_{N-k+1}\ldots\nu_N\nu_1\ldots\nu_k\\
R^{n+1}\widetilde{e} &= \overline{\nu_1\ldots\nu_{N-k}'\nu_{N-k+1}\ldots\nu_k}\nu_1\ldots\nu_{N-k}\nu_{1}\ldots\nu_k\\
        &= \overline{\nu_1\ldots\nu_{N-k}'\nu_{N-k+1}\ldots\nu_k}\nu_1\ldots\nu_{N-k}'\nu_{N-k+1}\ldots\nu_N\nu_1\ldots\nu_k\\
        &= \widetilde{e}\nu_{N-k+1}\ldots\nu_N\nu_1\ldots\nu_k
\end{align*}
Hence, we have shifted copies of the original backwards itineraries, and all folds have been congruent modulo $N$. As before, proceed by induction to achieve the desired result.
\end{proof}

\begin{figure}[!t]
	\centering

\begin{tabular}{c c c c c c } 
	$\tau$ & $\nu_1\ldots\nu_N$ & $k$& $e$ & $\widetilde{e}$ & Theorem \\ \hline
 			$(*12)^{\infty}$  &121&2 &$1^{\infty}21$ & $2^{\infty}$ & \ref{case2} \\
 			  &122 &2&$(1122)^{\infty}$ & $(1122)^{\infty}21$  & \ref{case3} \\
	$(*122)^{\infty}$ & 1221 &3& $(121)^{\infty}1221$&$2^{\infty}$ & \ref{case2}\\
	&1222&3&$(211222)^{\infty}$&$(211222)^{\infty}$221& \ref{case3}\\
	$(*112)^{\infty}$&1121&3&$(112)^{\infty}1112$&$(212)^{\infty}2122$& \ref{case2}	\\ 
				&1122&3&$(111122)^{\infty}$&$(111122)^{\infty}121$&\ref{case3}\\
	$(*1222)^{\infty}$&12221&4&$(1221)^{\infty}12221$&$2^{\infty}$&\ref{case2}\\
					&12222&4&$(22112222)^{\infty}$&$(22112222)^{\infty}2221$&\ref{case3}\\
	$(*1221)^{\infty}$&12211&4&$(1222)^{\infty}12211$&$(2221)^{\infty}2$&\ref{case2}\\
				&12212&3&$(121)^{\infty}12212$&$(112)^{\infty}11$&\ref{case2}\\
				&12211&3&$(112211)^{\infty}$&$(112211)^{\infty}212$&\ref{case3}\\
				&12212&4&$(22212212)^{\infty}$&$(22212212)^{\infty}2211$&\ref{case3}\\
	$(*1211)^{\infty}$&12111&4&$(12)^{\infty}111$&$(2211)^{\infty}2$&\ref{case2}\\
					&12112&3&$(122)^{\infty}12112$&$1^{\infty}$&\ref{case2}\\
					&12111&3&$(212111)^{\infty}$&$(212111)^{\infty}112$&\ref{case3}\\
					&12112&4&$(21212112)^{\infty}$&$(21212112)^{\infty}2111$&\ref{case3}\\
	$(*1122)^{\infty}$&11221&4&$(12111221)^{\infty}$&$(12111221)^{\infty}1222$&\ref{case3}\\
				&11222&4&$(12111222)^{\infty}$&$(12111222)^{\infty}1221$&\ref{case3}\\
	$(*1121)^{\infty}$&11211&3&$1^{\infty}211$&$(122)^{\infty}12$&\ref{case2}\\
				&11211&4&$(1122)^{\infty}11211$&$(212)^{\infty}12$&\ref{case2}\\
				&11212&3&$(111212)^{\infty}$&$(111212)^{\infty}211$&\ref{case3}\\
				&11212&4&$(12211212)^{\infty}$&$(12211212)^{\infty}1211$&\ref{case3}\\
	$(*1112)^{\infty}$ &11121&4&$1^{\infty}21$&$(2112)^{\infty}2$&\ref{case2}\\
					&11122&4&$(11111122)^{\infty}$&$(11111122)^{\infty}1121$&\ref{case3}  \\ \hline
\end{tabular}
	\caption{Admissible $\tau$ up to period 5 with backwards itineraries for Theorems \ref{case2} and \ref{case3}}
	\label{figurelist}
\end{figure}

\begin{theorem}\label{case3}
Suppose there exists a $N/2<k<N$ such that $\underline{d}(\sigma^{k}\tau,\tau)\geq N-k$, and let $\nu_1\ldots\nu_N \approx \tau_1\ldots\tau_N$, where $\nu_N \in \{1,2\}$ is chosen so that $\underline{d}(\sigma^{k}\nu,\tau)=N-k$. If
\begin{align*}
	 	e &=\overline{\nu_{N-k+1}\ldots\nu_k'\nu_1\ldots\nu_{N}} \text{ and}\\ 
	 	\widetilde{e} &=\overline{\nu_{N-k+1}\ldots\nu_k'\nu_1\ldots\nu_N}\nu_{N-k+1}\ldots\nu_N'\text{,} 
\end{align*}
then there exists asymptotic rays originating in $T(e)$ and $T(\widetilde{e})$. 
\end{theorem}

\begin{proof}
We begin by taking $\alpha_1(e)=N+1$ and $\alpha_1(\widetilde{e})=1$, and we obtain
\begin{align*}
Re &= \overline{\nu_{N-k+1}\ldots\nu'_k\nu_1\ldots\nu_N}\nu_{N-k+1}\ldots\nu_{k}\nu_{1}\ldots\nu_N\\
    &= \overline{\nu_{N-k+1}\ldots\nu'_k\nu_1\ldots\nu_N}\nu_{N-k+1}\ldots\nu_{k}\nu_{k+1}\ldots\nu_{N}'\nu_{N-k+1}\ldots\nu_N\\
R\widetilde{e} &= \overline{\nu_{N-k+1}\ldots\nu'_k\nu_1\ldots\nu_N}\nu_{N-k+1}\ldots\nu_N
\end{align*}
We then take the path to 
\begin{align*}
R^{n-1}e &= \overline{\nu_{N-k+1}\ldots\nu'_k\nu_1\ldots\nu_N}\nu_{N-k+1}\ldots\nu_k\nu_{k+1}\ldots\nu_N'\nu_1\ldots\nu_k\\
R^{n-1}\widetilde{e} &= \overline{\nu_{N-k+1}\ldots\nu'_k\nu_1\ldots\nu_N}\nu_1\ldots\nu_k
\end{align*}
Next, we may take $\alpha_n(e)=k+1$ and $\alpha_n(\widetilde{e})=N+k+1$, and we have
\begin{align*}
R^{n}e &= \overline{\nu_{N-k+1}\ldots\nu'_k\nu_1\ldots\nu_N}\nu_{N-k+1}\ldots\nu_{k}\nu_1\ldots\nu_{N-k}'\nu_1\ldots\nu_k\\
    &= e\nu_{N-k+1}\ldots\nu_N\nu_1\ldots\nu_k\\
R^{n}\widetilde{e} &= \overline{\nu_{N-k+1}\ldots\nu'_k\nu_1\ldots\nu_N}\nu_{N-k+1}\ldots\nu_k\nu_1\ldots\nu_N\nu_1\ldots\nu_k\\
    &= \overline{\nu_{N-k+1}\ldots\nu'_k\nu_1\ldots\nu_N}\nu_{N-k+1}\ldots\nu_N'\nu_{N-k+1}\ldots\nu_N\nu_1\ldots\nu_k\\
    &= \widetilde{e}\nu_{N-k+1}\ldots\nu_N\nu_1\ldots\nu_k
\end{align*}
Hence, we have left-shifted copies of the original itineraries, and all folds have been congruent modulo $N$. As before, proceeding by induction concludes the proof.

\end{proof}
\begin{proposition}
The backwards itineraries $e$ and $\widetilde{e}$ in each of Theorems \ref{case1}, \ref{case2} and \ref{case3} reside on distinct arc-components of $D_{\tau}$
\end{proposition}
\begin{proof}
With Proposition \ref{arccharacterization} in mind, it is easily seen that the backwards itineraries from Theorem \ref{case1} are on different arc-components. For Theorem \ref{case2}, observe that 
\begin{align*}
e &= \overline{\nu_1\ldots\nu_{N-k}\nu_{N-k+1}\ldots\nu'_k}\nu_1\ldots\nu_{N-k}\nu_{N-k+1}\ldots\nu_N \\
\widetilde{e} &=\overline{\nu_1\ldots\nu'_{N-k}\nu_{N-k+1}\ldots\nu_k}\nu_1\ldots\nu'_{N-k}\nu_{N-k+1}\ldots\nu'_N
\end{align*}
It is easily seen that the respective tails for $e$ and $\widetilde{e}$ ``line up," having infinitely many discrepancies, and neither being equivalent to $\tau$.

For Theorem \ref{case3}, after rewriting $e$ so that the tails ``line up," we have:
\begin{align*}
e &= \overline{\nu_{N-k+1}\ldots \nu_k \ldots \nu_N\nu_{N-k+1}\ldots\nu'_k\nu_1\ldots\nu_{N-k}}\nu_{N-k+1}\ldots\nu_N \\
\widetilde{e} &=\overline{\nu_{N-k+1}\ldots\nu_k'\nu_1\ldots\nu_N}\nu_{N-k+1}\ldots\nu_N'\
\end{align*}
Again, it is easily noticed that the tails are discrepant infinitely often, with neither being equivalent to $\tau$. 
\end{proof}
There are some noticable differences between this result, and the comparable result for $(I,f)$ presented in \cite{BAsymp}. For example, for $(I,f)$, the 
backwards itineraries may not be admissible. Moreover, the situations described in Theorems \ref{case2} and \ref{case3} are mutually exclusive in $(I,f)$. As shown in Figure~\ref{figurelist}, these may happen concurrently in $\hat{D}_{\tau}$. However, in \cite{BAsymp}, the author was able to give more information on the asymptotic structure of the asymptotic arc-components in $(I,f)$ (e.g., whether they form ``fans," ``cycles," or even combinations thereof). That $\hat{D}_{\tau}$ is not chainable makes analogous results difficult. The following corollaries address this line of inquiry.

\begin{corollary}
The backwards itineraries in Theorems \ref{case1}, \ref{case2}, and \ref{case3} give rise to a countably infinite collection of asymptotic rays.
\end{corollary}
\begin{proof}
In the inductive step, we may alter our choice for $V_{|A|}$.
\end{proof}

\begin{corollary}
The asymptotic arc-components in Theorem \ref{case1} form a $k$-fan (ie, the $k$ rays are pairwise asymptotic).
\end{corollary}
\begin{proof}
For each $j\in\{1,2,\ldots N-2\}$, let $\widetilde{e}^j=\overline{12^{N-2}1}12^{N-1}12^j$. As shown in Theorem \ref{case1}, each of the $\widetilde{e}^j$'s are asymptotic to a ray $\ray_j$ emmanating from $T(e)$, and we need only show that each $\ray_j$ is the same ray. The proof is similar to that of Theorem \ref{case1}. 
 
\end{proof}




Department of Mathematics

Baylor University

One Bear Place \#97328

Waco, TX 76798-7328

brent\_hamilton@baylor.edu

\begin{thebibliography}{9}

\bibitem{B1} S. Baldwin, ``Continuous Itinerary Functions and Dendrite Maps," \textit{Topology and its Applications} \textbf{154} (2007), 2889-2938.

\bibitem{B2} S. Baldwin, ``Inverse Limits of Tentlike Maps on Trees," \textit{Fund. Math.} \textbf{207} (2010), 211-254.

\bibitem{Ingram} M. Barge, H. Bruin, S. Stimac, ``The Ingram Conjecture," Preprint (2010).

\bibitem{BaDi} M.Barge, B.Diamond, ``Homeomorphisms of inverse limit spaces of one-dimensional maps,'' \textit{Fund. Math} \textbf{146} (1995), 171-187.

\bibitem{BDH} M. Barge, B. Diamond, C. Holton, ``Asymptotic orbits for primitive substitutions,'' \textit{Theoret. Comput. Sci.} \textbf{301} (2003) 439-450.

\bibitem{OIC} L.Block, S. Jakimovik, J. Keesling, ``On Ingram's Conjecture,'' \textit{Topology Proc.} \textbf{30} (2006), 95-114.

\bibitem{BBSub} K. Brucks, H. Bruin, ``Subcontinua of inverse limit spaces of unimodal maps,'' \textit{Fund. Math.} \textbf{160(3)} (1999), 219-246.

\bibitem{BBBook} K. Brucks, H. Bruin, \textit{Topics from one-dimensional dynamics}, LMS Student Text Series, Cambridge University Press (2004).

\bibitem{BrDi} K.Brucks, B.Diamond, ``A symbolic representation of inverse limit spaces for a class of unimodal maps,"  Continuum Theory and Dynamical Systems, in: \textit{Lecture Notes in Pure and Appl. Math.}, vol. 149, Dekker, New York, 1995.

\bibitem{BEmb} H.Bruin, ``Planar embeddings of inverse limit spaces of unimodal maps," \textit{Topology and its applications} \textbf{96} (1999) 191-208.  

\bibitem{BAsymp} H.Bruin, ``Asymptotic arc-components of unimodal inverse limit spaces," \textit{Topology and its applications} \textbf{152} (2005) 182-200. 

\bibitem{KPer} L. Kailhofer, ``A classification of inverse limit spaces of tent maps with periodic critical points,'' \textit{Fund. Math.} \textbf{177} (2003), 95-120.

\bibitem{RaStiSt} B. Raines, S. \v{S}timac, ``Structure of inverse limit spaces of tent maps with nonrecurrent critical points,'' \textit{Glas. Mat.}, \textbf{42(62)}, (2007), 43-56.

\bibitem{RaStiNR} B. Raines, S. \v{S}timac, ``A classification of inverse limit spaces of tent maps with non-recurrent critical point,'' \textit{Algebraic and Geometric Topology} \textbf{9} (2009), 1049-1088.

\bibitem{StSt} S. \v{S}timac, ``Structure of inverse limit spaces of tent maps with finite critical orbit,'' \textit{Fund. Math.} \textbf{191} (2006), 125-150.

\bibitem{SPreper} S. \v{S}timac, ``A classification of inverse limit spaces of tent maps with finite critical orbit,'' \textit{Topology Appl.} \textbf{154} (2007), 2265-2281.


\end{thebibliography}
 \end{document}